\documentclass[12pt]{article}
\usepackage{amsmath,amssymb,amsfonts,amsthm,url,xspace,graphicx,enumerate}
\newtheorem{theorem}{Theorem}
\numberwithin{theorem}{section}
\newtheorem{conjecture}{Conjecture}
\newtheorem{lemma}[theorem]{Lemma}
\newtheorem{proposition}[theorem]{Proposition}
\newtheorem{prop}[theorem]{Proposition}

\newtheorem{definition}[theorem]{Definition}

\theoremstyle{remark}
\newtheorem*{unremark}{Remark}

\newtheorem{example}[theorem]{Example}

\newcommand{\Z}{\mathbb{Z}}
\newcommand{\R}{\mathbb{R}}
\newcommand{\HH}{\mathbb{H}}

\newcommand{\C}{\mathbb{C}}

\newcommand{\B}{{\cal B}}
\newcommand{\eps}{\varepsilon}
\newcommand{\be}{\begin{equation}}
\newcommand{\ee}{\end{equation}}
\newcommand{\old}[1]{}

\newcommand{\tr}{\text{tr}}
\newcommand{\qdet}{\text{qdet}}

\usepackage[usenames,dvipsnames]{color}
\newcommand{\red}[1]{{\color{red} #1}}
\newcommand{\td}[1]{{\red{\textbf {[#1]}}}}

\def\noproof{\hfill \Box}
\def\init{{\cal I}}

\def\tilings{{\cal N}}
\def\disp{\displaystyle}

\begin{document}

\title{Principal minors and rhombus tilings}
\author{Richard Kenyon\thanks{The research of RK was supported by NSF Grant DMS-1208191}\\ Brown University\\ Providence, RI 02912, USA \and 
Robin Pemantle\thanks{The research of RP was supported by NSF Grant DMS-1209117}
\\ University of Pennsylvania\\ Philadelphia, PA 19104, USA
}
\maketitle

\begin{abstract}
The algebraic relations between the principal minors of a generic 
$n\times n$ matrix are somewhat mysterious, see e.g.~\cite{lin-sturmfels}.  
We show, however, that by adding in certain \emph{almost} principal minors,
the ideal of relations is generated by translations of a single relation, 
the so-called hexahedron relation, which is a composition of six cluster 
mutations.

We give in particular a Laurent-polynomial parameterization of the space
of $n\times n$ matrices, whose parameters consist of certain principal
and almost principal minors.  The parameters naturally live on vertices
and faces of the tiles in a  rhombus tiling of a convex $2n$-gon.
A matrix is associated to an equivalence class of tilings,
all related to each other by Yang-Baxter-like transformations.

By specializing the initial data we can similarly parametrize the space 
of Hermitian symmetric matrices over $\R, \C$ or $\HH$ the quaternions.
Moreover by further specialization we can parametrize the space of 
\emph{positive definite} matrices over these rings.  
\end{abstract}

\section{Introduction}

Classical statistical mechanical models related to free fermions, 
such as the Ising model and dimer model on planar graphs, are solved 
by determinantal methods.  In these models the energy correlations or 
edge correlations are determinantal, in the sense that they are computed 
as (principal) minors of an underlying matrix kernel~\cite{Keny2009}.
The Yang-Baxter equation for the Ising model, and its analogue for 
the dimer model (the ``urban renewal transformation") both arise from 
algebraic identities among the minors of the corresponding matrix kernel. 
One would like to understand, for a model with ``generic" interactions, 
in what sense these identities are the \emph{only} algebraic identities 
among the correlations.  For these reasons it is of interest 
to study in an abstract setting the algebraic relations between 
the principal minors of a matrix.  Remarkably, once one adds in 
``almost" principal minors, we arrive at a complete description 
in terms of the translates of a single relation (the hexahedron 
relation).  Furthermore our description is two-dimensional, in the 
sense that sets of free parameters naturally lie on planar networks 
(rhombus tilings of polygons).

A {\em principal minor} of a complex $n \times n$ matrix $M$ is
the determinant of a submatrix $M_A^A$ where $A$ is a subset of
$[n] := \{ 1 , \ldots , n \}$ and $M_A^B$ denotes the submatrix
of $M$ obtained by restricting rows to $A$ and columns to $B$.
There are $2^n$ principal minors of $M$ if one includes the 
trivial minor $\det M_\emptyset^\emptyset := 1$.  Introducing an
indeterminate $x_A$ for each nontrivial minor, one may ask what 
polynomial relations hold among the minors, that is, what polynomials 
in $\C [x_A : A \subseteq [n]]$ always vanish.  

If one restricts attention to symmetric matrices, the answer 
is reasonably nice.  The ideal of relations among the principal 
minors of a symmetric matrix of rank~4 is given in~\cite{holtz-sturmfels}
and conjectured for all ranks; at the set-theoretic level, this was 
proved by Oeding~\cite{Oedi2011}.
The algebraic relations between these principal minors of a general
$n \times n$ matrix are, by contrast, somewhat mysterious.  For example, 
when $n=4$, Lin and Sturmfels~\cite{lin-sturmfels} show that the ideal 
of all polynomial relations is minimally generated by 65~polynomials 
of degree~12 (see also the statement without proof in~\cite{BoRa2005}).
Part of the complication occurs because certain
vectors of values $\{ x_A \}$ that have a lot of zeros 
cannot occur as collections of minors even though they satisfy 
the relations in a sort of vacuous way.  Looking only at {\em generic} 
vectors of values, namely those in $\disp (\C \setminus \{ 0 \})^{2^n}$,
ameliorates this problem.  The complex torus $\disp (\C \setminus 
\{ 0 \})^{2^n}$ is coordinatized by the Laurent algebra 
$\C [x_A , x_A^{-1} : A \subseteq [n]]$, over which we will 
work henceforth.  

Say that $\det M_A^B$ is an {\em almost-principal minor} if $A, B 
\subseteq [n]$ with $|A| = |B|$ and if the sets differ by swapping
only one element: there are distinct elements $x \in A$ and $y \in B$ 
such that $A \setminus \{ x \} = B \setminus \{ y \}$.  Note that in
our lingo, a principal minor is not almost principal; this differs
from the convention in, for example,~\cite{Stur2008}.
Divide the almost-principal
minors into two classes, say {\em odd} and {\em even}, by putting
$M_A^B$ in the odd class if $A = S \cup \{ i \}$ and $B = S \cup \{ j \}$
with $(i-j) (-1)^{|S|} > 0$.  In other words, if the extra row index
is greater than the extra column index then the parity of the minor 
is the same as the parity of $S := A \cap B$, but when the extra
column index is greater than the extra row index, then the parity 
of the minor is opposite to the parity of $S$.

Our first result concerns the relations that hold among the
principal minors and the odd almost-principal minors (by symmetry 
we could use even almost-principal minors instead).  
Arranging the subsets of $[n]$ on a Boolean lattice, we show
that the ideal in the Laurent algebra of relations holding
among all generic vectors of principal and almost-principal minors
is generated by (lattice) translates of a single polynomial relation, 
the so-called hexahedron relation of~\cite{KP-ising}.  The hexahedron 
relation is a set of four polynomial relations holding among fourteen
variables indexed by the eight vertices and six faces of a cube.
For any Boolean interval $[S , S \cup D]$ of rank three in the
Boolean lattice $\B_n$ of rank $n$, the vertices and faces may be 
naturally associated with the eight principal and six odd almost 
principal minors of the form $\disp{\{ \det M_{S \cup A}^{S \cup B} \, : 
\, A , B \subseteq D \}}$.  As $[S , S \cup D]$ vary over rank-3
Boolean intervals in $\B_n$, the corresponding hexahedron relations
generate the prime ideal of all polynomial relations among these minors.

This result is stated as Theorem~\ref{th:ideal} below.  It implies the 
weaker notion of set-theoretic generation of all relations.  Philosphically
we interpret this as follows.  The ideal of relations among principal
minors is complicated when represented directly.  However, its tensor
with the Laurent algebra is the intersection of the ideal generated 
by hexahedron relations with the subring of Laurent polynomials in 
only the principal minor variables.  

In the terminology of cluster algebras (see, e.g.,~\cite{FoZe2002}),
the hexahedron relation is a composition of six cluster mutations.
These are instances of the so-called {\em urban renewal} 
transformations invented by Kuperberg and described in~\cite{Ciuc1998}.
This allows us explicitly to parameterize the variety of 
all possible collections of principal and odd almost principal
minors.  One must first pick a set of variables $x_A^B$, 
call these the {\em initial conditions}, to specify, and then 
describe the remaining variables in terms of the initial conditions.
There are many ways of choosing the set of variables for the
initial conditions.  These correspond naturally
to the rhombus tilings of a $2n$-gon (see Section~\ref{sec:2n-gon} below).  
For any fixed tiling,
the matrix entries and all principal and odd almost principal
minors are Laurent polynomials in the initial variables
associated with the chosen tiling.  This is Theorem~\ref{th:laurent}
below.  Given the cluster algebra representation, this comes as no
surprise because compositions of mutations in a cluster algebra 
preserve the property of being a Laurent polynomial (see~\cite{FZ});
as one varies the tiling, the associated variables are related by 
Yang-Baxter-like transformations preserving the Laurent property.

In the last part of the paper we specialize the initial data
to subclasses, obtaining parameterizations for certain subclasses
of matrices.  We parameterize the class of Hermitian matrices,
and restricting to $\R$, the class of real symmetric matrices.
This is Theorem~\ref{th:hermitian} below. In this case the 
hexahedron relation specializes to the so-called \emph{Kashaev relation} 
which arises in the Yang-Baxter relation for the Ising model, 
see \cite{KP-ising}.  We also extend to 
a non-commutative setting and parametrize the 
quaternion-Hermitian matrices (Theorem~\ref{th:quat} below).
Moreover by further specialization we can parametrize the space of 
\emph{positive definite} matrices over these rings 
(Theorem~\ref{positive} below).  This is a positive description 
in the sense that the entries are positive Laurent polynomials 
in the parameters, satisfying interval constraints.

\section{The hexahedron relation and matrix minors}

Let $a$ be a function on the set of vertices and faces of a cube.
Label the vertices and faces of a cube by indices~0 thorugh~9
and $0^*, 1^*, 2^*$ and $3^*$ so that the values of $a$, 
denoted by $a_0, \ldots , a_9, a_0^*, \ldots , a_3^*$ are
arranged on the cube as in Figure~\ref{cubemove}.  

\begin{figure}[ht!]
\begin{center}\includegraphics[width=4in]{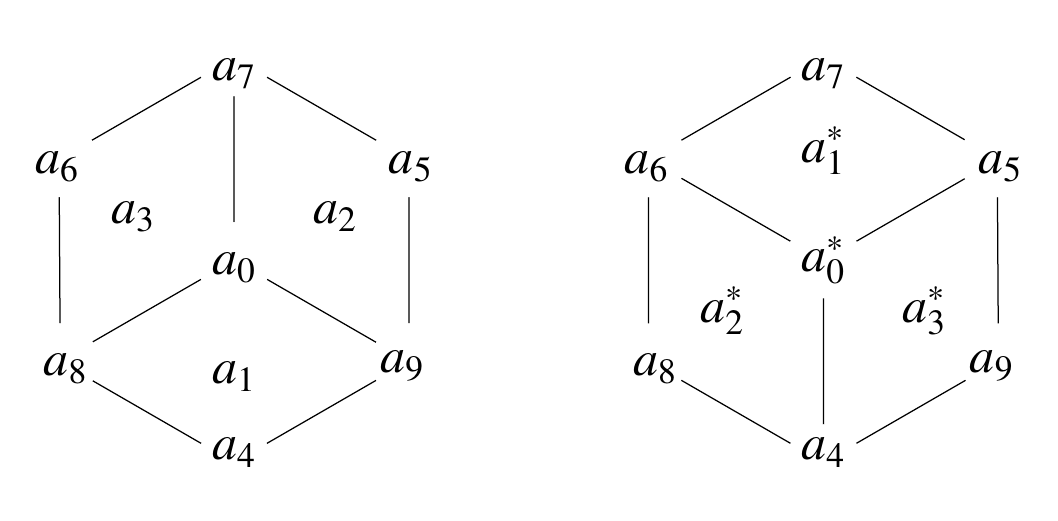}
\\ Valley view \hspace{1.6in} Peak view \end{center}
\caption{The variables in the hexahedron relation}
\label{cubemove}
\end{figure}

The function $a$ is said to satisfy the \emph{hexahedron relation} on this 
cube if the following four polynomial identities hold.  

{\small \begin{eqnarray}\label{hh1}
a_1^*a_1a_0&=&a_1a_2a_3+a_7a_8a_9+a_0a_4a_7\\
\label{hh2}a_2^*a_2a_0&=&a_1a_2a_3+a_7a_8a_9+a_0a_5a_8\\
\label{hh3}
a_3^*a_3a_0&=&a_1a_2a_3+a_7a_8a_9+a_0a_6a_9\\
a_0^*a_0^2a_1a_2a_3&=&
(a_1a_2a_3)^2+a_1a_2a_3(2a_7a_8a_9+\nonumber
a_0a_4a_7+a_0a_5a_8+a_0a_6a_9)+\\&&\label{hh4}
+(a_8a_9+a_0a_4)(
a_9a_7+a_0a_5)(a_7a_8+a_0a_6).
\end{eqnarray}}

Note that the relation is symmetric under cyclic rotation around the 
$a_0a_0^*$ axis; one can check that this relation is also ``top-down" 
symmetric: symmetric under the reversal 
$$a_0^*\leftrightarrow a_0,~a_1^*\leftrightarrow a_1,~a_2^*
   \leftrightarrow a_2,~a_3^*\leftrightarrow a_3,~a_4\leftrightarrow a_7,~a_5
   \leftrightarrow a_8,~a_6\leftrightarrow a_9.$$
This relation was introduced in~\cite{KP-ising}, where the cube
was taken to vary over cells of the cubic lattice $\Z^3$ and
the hexadron relations taken to define translation invariant 
relations on a function on vertices and faces of the cubic lattice.
The relations were shown there to be compositions of six cluster 
mutations.  Initial conditions in this case correspond to stepped
surfaces in the cubic lattice and the cluster structure implies
that all variables are Laurent polynomials in any set of initial
variables.  
\old{
Also there it was shown that the hexahedron relations extend the 
Kashaev relations~\cite{Kash1996} for functions on the the vertices 
of $\Z^3$.  However, where the Kashaev relations are quadratic, 
the hexahedron relations are linear in the variables 
$a_0^*, a_1^*, a_2^*$ and $a_3^*$.
The relations may be therefore used as a recurrence defining the values 
$a_0^*, a_1^*, a_2^*$ and $a_3^*$ in terms of the values 
$a_0 , \dots , a_9$.  Whereas the Kashaev recurrence is quadratic,
hence branching, the hexahdron recurrence is well defined as long as
as long as $a_0a_1a_2a_3 \neq 0$.  Embedding the Kashaev relations
in the hexahedron relations by adding face variables thus simplifies
the recurrence considerably.
}

In the present work, we show that the hexahedron relation is the
relation satisfied by the minors of a matrix.  This requires placing
the hexahedron relations on the Boolean lattice $\B_n$ 
(the $n$-cube $\{0,1\}^n$ with its natural partial order)
in place of the cubic lattice $\Z^3$.  We do so by allowing the cube 
in Figure~\ref{cubemove} to vary over Boolean intervals of rank~3 
in the rank-$n$ Boolean lattice.  We do this in a way that obtains
the picture in Figure~\ref{corr2}, which we now explain.

\begin{figure}[ht!]
\begin{center}\includegraphics[width=4in]{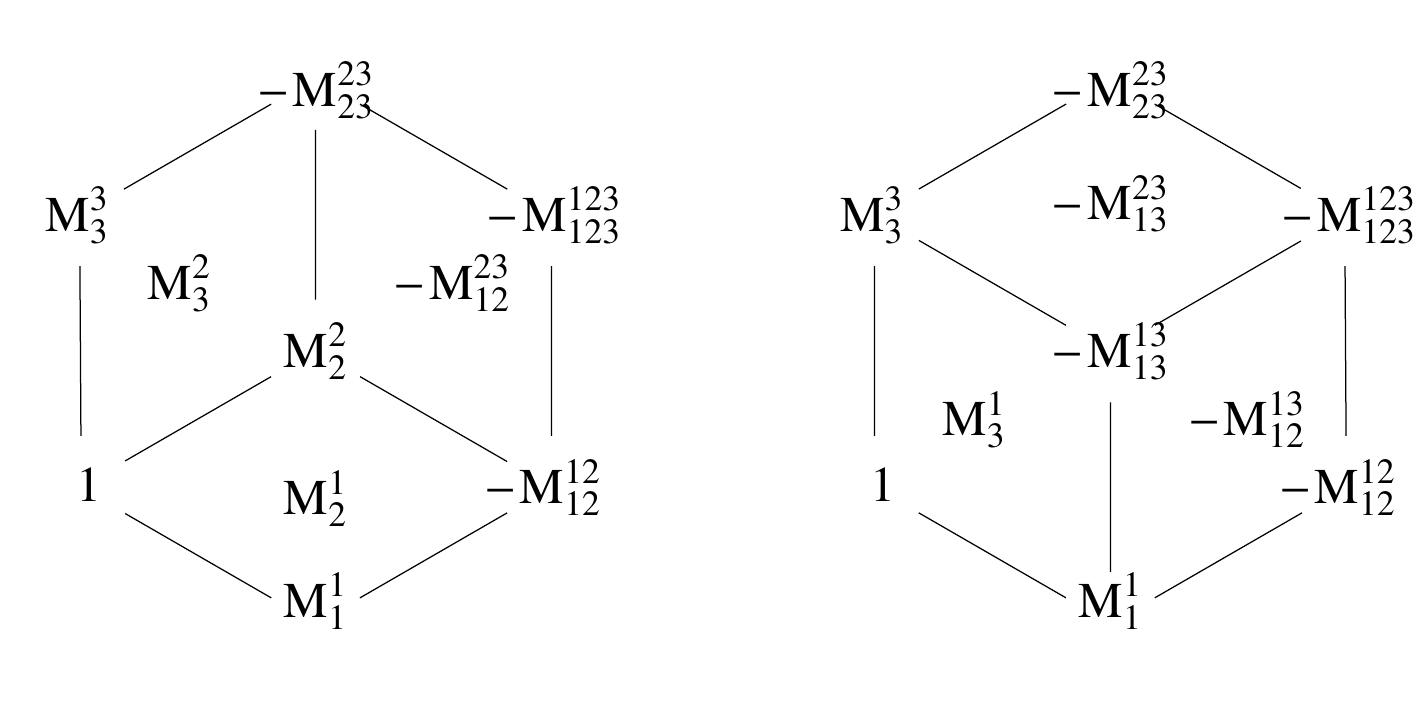}\end{center}
\caption{Arrangement of principal and near principal minors on $\B_3$}
\label{corr2}
\end{figure}

Fix $n$ and an interval $\init := [S , S \cup D]$ 
of rank three in $\B_n$.  The order preserving bijection $\alpha$
between $\{ 1, 2, 3\}$ and $D$ induces a bijection between $\B_3$ and $\init$,
explictly $\alpha_* (A) = (S \cup \alpha[A])$.  The notation for 
matrix minors becomes less cluttered if we use $M_A^B$ in place
of $M_{\alpha*(A)}^{\alpha_*(B)}$ when the set $S$ can be understood.
Using this abbreviation, write the principal minor $M_A^A$ at the
element $A \in \B_3$.  Interpreting the Hasse diagram of $\B_3$ as
a cube, each of the six faces  is a rank-2 interval.  If $A$ and $B$
are the two middle-rank elements of such an interval, then associate
with the corresponding face the almost principal minor $M_A^B$ or $M_B^A$, 
choosing whichever one of these is odd.  We now have a set of 
eight principal and six odd almost principal minors of $M$ associated
with the eight vertices and six faces of $\B_3$.  We need to change
the signs of seven of these, namely the vertices of rank~2 and~3
and the upper faces.  Invoking the hexahedron
recurrence is now a matter of matching to Figure~\ref{cubemove}, which
we do in a slightly non-intuitive manner, matching $M_\emptyset^\emptyset$
to $a_8$, $M_1^1, M_2^2$ and $M_3^3$ to $a_4, a_0$ and $a_6$ respectively,
and so on (there is only one way to extend this graph isomorphism to
the whole cube).  The result is Figure~\ref{corr2}.

\begin{lemma}\label{hexminorcorr}
Under the correspondence between the diagrams in Figures~\ref{cubemove}
and~\ref{corr2}, the minors of $M$ satisfy the hexahedron relation. 
\end{lemma}

\begin{proof}
When $n=3$, the only choice for $S$ is $S = \emptyset$ and the
abbreviation and actual notation $M_A^B$ coincide.  In this case
the proof is a quick algebraic verification.
Muir's law of extensible minors~\cite{muir1883}, states that 
``a homogeneous determinantal identity for the minors of a matrix 
remains valid when all the index sets are enlarged by the same 
disjoint index set.''  (See~\cite{berliner-brualdi} for this wording 
and~\cite[Section~7]{brualdi-schneider} for a proof).  Here, 
homogeneity means that every monomial in the identity is a product 
of determinants of degrees summing to the same value.  In the first 
three hexahedron identities \eqref{hh1}--\eqref{hh3} every monomial 
has degree~4, while in~\eqref{hh4}, every monomial has degree~8.  
The conclusion of  Muir's law is the conclusion of the lemma.
\end{proof}

\section{$2n$-gon networks} \label{sec:2n-gon}

On the cubic lattice, initial conditions are stepped surfaces,
with moves from one stepped surface to another corresponding to
the addition or removal of a cube.  The Boolean lattice is a 
cell complex and although its dimension is not~3, addition and
removal of a 3-cube still represents a well defined family of
moves between 2-chains in a family of 2-chains sharing a common
boundary.  These two-chains, which correspond to initial conditions,
are described by tilings of a $2n$-gon, as we now describe.  One
of these tilings is called the {\em standard tiling} and is shown
in Figure~\ref{4genexample} (ignore the dotted lines for now).

\begin{figure}[ht!]
\begin{center}\includegraphics[width=2.5in]{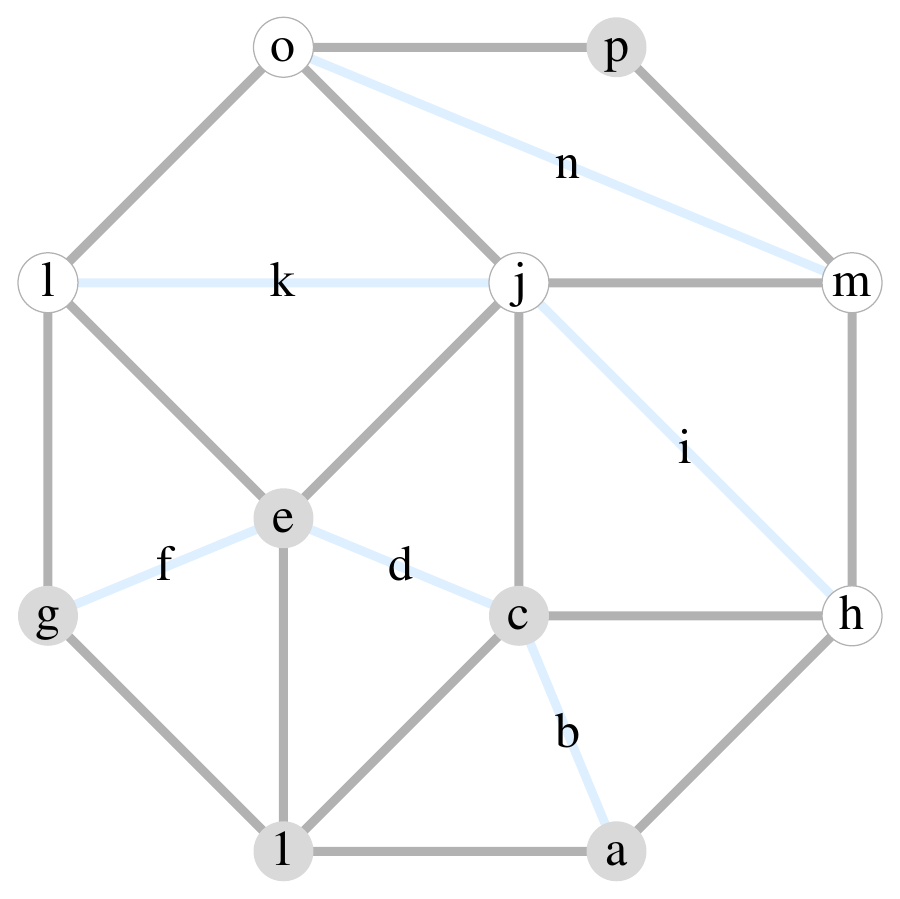}\end{center}
\caption{\label{4genexample}The standard network.  White vertices 
will be those with $\sigma(v)=-1$.}
\end{figure}

Let $P_n$ be the regular $2n$-gon with unit length edges, 
oriented so that it has a horizontal edge.  Let $v_0$ be the 
vertex of $P_n$ which is the left endpoint of the lower 
horizontal edge. Place $P_n$ so that $v_0$ is at the origin 
in $\R^2$.  The polygon $P_n$ is the projection to the plane 
of the $n$-cube $[0,1]^n$ with the property that for each 
$j \in [n]$, the basis vector ${\mathbf e}_j$ projects to the vector 
$e_j := e^{\pi i(j-1)/n}$.

The tilings of $P_n$ we consider are tilings by translations 
of the set ${\cal W}_n$ of tiles, where ${\cal W}_n$ is the set 
of rhombi $R_{jk}$ with unit edges parallel to two distinct edges 
$e_j$ and $e_k$ of $P_n$.  
The set ${\cal W}_n$ has cardinality $\disp{\binom{n}{2}}$.  
Each tile in ${\cal W}_n$ occurs precisely once in each tiling.  It may not be
obvious that there exist such tilings (or even that the areas of tiles 
in ${\cal W}_n$ sum to the area of $P_n$) but the following construction 
of the standard tiling shows there to be at least one such tiling.
Define the \emph{standard tiling} $T_0$ by placing all rhombi
$R_{i,i_+1}$ with their lowest point at the origin (in the case of
$R_{12}$, the leftmost lowest point).  In the $n-2$ gaps between
the uppermost extensions of these, place the rhombi $R_{i,i+2}$,
and continue this way until the rhombus $R_{1n}$ is placed, 
filling the last hole in $P_{2n}$.  This tiling has the property
that the vertices are precisely the points $v_0 + \sum_{j \in G} e_j$ 
for some set $G$ of consecutive elements of $[n]$.

A ``cube move'' consists in taking three tiles of $T$ whose union 
is a hexagon and replacing them with the three same tiles in the 
other order, effectively ``pushing" the tiling across a $3$-cube.
Lifting back to $\B_n$, one sees that all the 2-chains have the same
boundary, which is the lifting of the boundary of $P_n$ to $\B_n$.
Each vertex $v$ of the tiling lifts to a lattice point in $\B_n$, 
which is the sum of $e_j$ for all $j$ such that $e_j$ is on 
the path from the origin to $v$ using edges of the tiling.  The
space of all tilings of $P_n$ by ${\cal W}_n$ is connected under 
cube moves: see~\cite{keny1993}.

\subsubsection*{Labeled tilings}

A \emph{$2n$-gon network} is a labeled $2n$-gon tiling.  Formally,
this means it is a pair $(T,F)$ where $T$ is a tiling and $F$ is
a real or complex function on $U(T)$, the set of faces and vertices
of $T$.  (In the last section we consider quaternionic networks and
matrices.)

Two networks are \emph{equivalent} if one can be obtained from the 
other by a sequence of cube moves, in which the tiles are replaced 
by a cube move and the vertex and face values undergo a hexahedron 
transformation, meaning that the values on the center vertex 
$a_0$ and the faces $a_1,a_2,a_3$ are transformed to 
$a_0^*,a_1^*,a_2^*,a_3^*$ on the new network or vice versa.
We also allow as an equivalence move multiplication of all 
values by a single nonzero constant; the hexahedron relations
are homogeneous, hence always preserved by such scaling.
We say that a $2n$-gon network is {\em generic} if it and all 
equivalent networks have only nonzero labels.  

\begin{prop}\label{multi}
The equivalence class of a generic network contains precisely
one network $(T,F)$ for a given $T$ such that $F(v_0) = 1$.
\end{prop}

In other words, if two sequences of cube moves lead to the same tiling, 
then the resulting network does not depend on the sequence of cube 
moves leading to it.  The proof will follow from Theorem \ref{nettomtx} 
below; see the remarks after the proof of that theorem. 

\section{Correspondence between matrices and networks}

Let $M_n^* (\C)$ denote the set of generic $n \times n$ complex matrices,
meaning those with only nonzero minors.  Let $\tilings$ denote the set of 
generic $2n$-gon networks.  In this section we describe a map 
$\beta_T$ of the form $A \mapsto (T , F_{A,T})$ and a map 
$\Psi : \tilings \to M_n^* (C)$ that together establish a bijection 
between $M_n^* (\C)$ and equivalence classes in $\tilings$.

\subsection{Matrices to networks}

Let $A \in M_n^* (\C)$ be a matrix and $T$ a tiling of the $2n$-gon.
For a vertex $v$ of $T$, define $\sigma(v) = (-1)^{\lfloor d/2\rfloor}$ 
where $d$ is the graph distance in the tiling from $v$ to $v_0$.
Recall that $U(T)$ denotes the union of the vertices and faces of $T$.  
Define a function $F = F_{A,T}$ on $U(T)$ as follows.
Each vertex $v$ of $T$ is naturally associated with a point in 
$\B_n$, that is, a subset $S \subseteq [n]$.  Let
$F(v) = \sigma(v)\det A_S^S$ where $A_S^S$ is the principal minor 
of $A$ indexed by $S$.  On a rhombus $R_{ij}$ with vertices 
$v,v+e_i,v+e_i+e_j,v+e_j$ and $i<j$ we assign the value
\begin{equation} \label{eq:F}
F(R_{ij}) = \sigma(v) \det K_{S\cup\{j\}}^{S\cup\{i\}}
   ~~\mbox{ or }~~ \sigma(v) \det K_{S\cup\{i\}}^{S\cup\{j\}} \, , \;
   \mbox{ whichever is the odd minor};
\end{equation}
here again $S$ is the subset of $[n]$ corresponding to $v$.

\begin{theorem}\label{mtxnetthm} For any tilings $T$ and $T'$
the networks $(T , F_{A,T})$ and $(T' , F_{A,T'})$ are 
equivalent.  Consequently the map $(A,T) \mapsto (T,F_{A,T})$
induces a function $\Phi$ mapping each matrix $A \in M_m^* (\C)$
to the equivalence class of $(T , F_{A,T})$, which does not
depend on $T$.
\end{theorem}

\begin{proof}
Suppose $T$ and $T'$ differ by a cube move.  The functions $F_{A,T}$
and $F_{A,T'}$ label the vertices and faces according to the diagrams
in Figure~\ref{corr2} (the matrix is now named $A$ rather than $M$ 
and we use the convention that $[S , S \cup D]$ is mapped in the
order preserving way to $\B_3$).
By Lemma~\ref{hexminorcorr}, these two functions are related by a
hexahedron relation and are thus by definition equivalent.  Any 
two tilings are connected by a finite sequence of cube moves, hence 
$(T , F_{A,T})$ and $(T' , F_{A,T'})$ are equivalent for any $T, T'$.  
Genericty of $A$ implies $F_{A,T}$ is nowhere zero, which proves 
genericity of the equivalence class of $(T , F_{A,T})$.
\end{proof}

\begin{unremark}
This implies Proposition~\ref{multi} for networks in the range of $\Phi$.
\end{unremark}

\subsection{Networks to matrices}

Conversely, let us explain how to go from an equivalence class of 
generic networks to a matrix in $M_n^* (\C)$.  A network $(T,F)$ is
called {\em standard} if $T = T_0$ and $F(v_0) = 1$.  The key step is
to construct the map $\Psi$ taking a standard network $(T_0,F)$ to
a matrix $A$ such that $F_{T_0,A}$ agrees with $F$ on $U(T_0)$.  

Our strategy will be to assign matrix entries $A_{ij}$ in a particular 
order so that we can check inductively that the assigned entries force
$F_{A,T_0}$ to agree with $F$ on ever larger subsets of $T_0$
no matter what the values of the yet unassigned entries of $A$.
In this way we both construct $A$ and verify that $F_{A,T_0} = F$.
Visualization is easy when working with $T_0$ because each vertex 
and rhombus corresponds to a contiguous subdeterminant, 
the values of $F_{A,T_0}$ at vertices being principal minors 
of the form $\det A_{i,i+1,\ldots ,j}^{i,i+1,\ldots ,j}$ 
and the values at rhombi being odd almost-principal minors of 
the form $A_{i, \ldots , j}^{i\pm1, \ldots , j\pm 1}$.

Before giving a formal description we illustrate with an example
where $n=4$.  Figure~\ref{4genexample} shows the standard tiling of
$P_8$ with vertices and rhombi labeled by indeterminates.  
If $A$ is a matrix with $F_{T_0,A} = F$ then reading values 
of $F$ on $U(T_0)$ along successive dotted paths, starting 
from the lower right, determines successive minors of $A$
as follows.  The first dotted path contains $1 \times 1$ 
minors, therefore dictating the matrix entries 
$A_{11}, A_{21},A_{22},A_{32}, A_{33}, A_{43}$ and $A_{44}$
(notation: two subscripts for matrix entries, whereas one subscript
and one superscript for the corresonding $1 \times 1$ minor).  
The second dotted path, read right to left, gives the negatives of the minors
$A_{12}^{12},A_{23}^{12},A_{23}^{23},A_{34}^{23}$ thereby
determining $A_{12}, A_{23}, A_{34}, A_{31}$ and $A_{42}$.
The third dotted path, read right to left, gives the negatives 
of the minors $A_{123}^{123},A_{234}^{123}$ and $A_{234}^{234}$,
thereby determing $A_{13}, A_{24}$ and $A_{41}$.  The last
value is $\det A$, which now determines $A_{14}$, all other
entries of $A$ already having been determined.
Explicitly, in terms of the indeterminates labeling the
vertices and faces in Figure~\ref{4genexample}, the matrix is
given by
{\small
$$\begin{pmatrix}
 a & \frac{a c}{b}+\frac{h}{b} & \frac{h e}{b d}+\frac{a c e}{b d}+\frac{h j}{b c d}+\frac{a j}{b d}+\frac{h j}{c i}+\frac{m}{i} & X \\[1ex]
 b & c & \frac{c e}{d}+\frac{j}{d} & \frac{j g}{d f}+\frac{c e g}{d f}+\frac{j l}{d e f}+\frac{c l}{d f}+\frac{j l}{e k}+\frac{o}{k} \\[1ex]
 \frac{b d}{c}+\frac{i}{c} & d & e & \frac{e g}{f}+\frac{l}{f} \\[1ex]
 \frac{i f}{c e}+\frac{b d f}{c e}+\frac{i k}{c d e}+\frac{b k}{c e}+\frac{i k}{d j}+\frac{n}{j} & \frac{d f}{e}+\frac{k}{e} & f & g \\[1ex]
\end{pmatrix}
$$
}
where $X=$
{\small
\begin{eqnarray*}
&&
\frac{a c e g}{b d f}+\frac{a c l}{b d f}+\frac{a j l}{b d e f}
   + \frac{a g j}{b d f}+\frac{a j l}{b e k}+\frac{a o}{b k}
   + \frac{h j l}{b c d e f}+\frac{g h j}{b c d f}
   + \frac{h j l}{b c e k}+\frac{h o}{b c k}+\frac{e g h}{b d f} \\[2ex]
& + & \frac{h l}{b d f}+\frac{d h j l}{c e i k}+\frac{d h o}{c i k}
   + \frac{h j l}{c e f i}+\frac{g h j}{c f i}+\frac{d l m}{e i k} 
   + \frac{d m o}{i j k}+\frac{l m}{e f i}+\frac{g m}{f i}+\frac{m o}{j n}
   + \frac{p}{n}.
\end{eqnarray*}
}

To see why this works in general,
divide $U(T_0)$ into disjoint paths, each alternating between
vertices and rhombi.  The zeroth path is the vertex $v_0$;
the first path contains the vertices at distance~1 from $v_0$
and the rhombi $R_{i,i+1}$ between them.  The $j^{th}$ path contains
the vertices at distance $j$ from $v_0$ and the rhombi between them.
This partition is illustrated by the dotted paths in Figure~\ref{4genexample}.

Vertices on the $j^{th}$ path, $j \geq 1$, will induce assignments 
of elements of $A$ on the $(j-1)^{st}$ superdiagonal, where the 
zeroth superdiagonal is the main diagonal.  Rhombi on the $j^{th}$ 
path will induce assignments of elements of $A$ on the $j^{th}$ 
subdiagonal.  The zeroth path always contains the element~1,
so provides no new information and does not induce an assignment.

Inductively, we check that for each vertex or rhombus, the
equation that $F_{A,T_0}$ agrees with $F$ at each new face or vertex,
which is the equation $\det A_S^T = c$ for some $S, T \subseteq [n]$ 
and some number $c$.  This is a multilinear equation with precisely 
one unassigned variable.  Indeed, for vertices in the $j^{th}$ path 
it is a specification of a contiguous subdeterminant spanning from 
the diagonal to the $(j-1)^{st}$ superdiagonal while for rhombi 
on this path it is a specification of a contiguous subdeterminant 
spanning from the first subdiagonal down to the $j{th}$ subdiagonal. 

One of these linear equations is degenerate if and only if
the cofactor of that determinant vanishes.  The cofactor is
the value of $F$ at a position one row closer to the main diagonal.
Genericity of $(T_0 , F)$ implies that this is nonzero.  This completes 
the induction.  We conclude there is a unique matrix $A$ for which 
$F_{A,T_0}$ agrees with $F$; we call this $\Psi (T_0 , F)$.  
We have now proved the first and only nontrivial statement in 
the following theorem.

\begin{theorem}\label{nettomtx} 
If $(T_0 , F)$ is generic then there is a unique $A \in M_n^* (\C)$
such that $F_{A,T_0} = F$ on $U(T_0)$.  The map $A \mapsto 
(T_0 , F_{A,T_0})$ and the map $\Psi$ mapping
$(T_0,F)$ to $A$ are two-sided inverses. 
\end{theorem}

\begin{proof}
The construction always produces a matrix $A$ such that
$F_{A,T_0}$ agrees with the given $F$.  If $(T,F)$ and $(T',F')$
are related by a cube move and $F_{A,T} = F$ then $F_{A,T'} = F'$
because the hexahedron hold for the minors of $A$.  Therefore,
if $(T,F)$ is equivalent to $(T,F')$ then $F = F'$.  This
proves there is only one network $(T,F)$ in each equivalence class,
implying Proposition~\ref{multi}, ensuring that $\Psi$ is well defined,
and proving that $A \mapsto (T_0 , F_{A,T_0})$ and $\Psi$ are inverses.
\end{proof} 

\subsection{Further properties of the correspondence}

Each matrix $M \in M_n (\C)$ has a vector of $2^n$ principal minors.
Let $\C^*$ denote $\C \setminus \{ 0 \}$.  Let $V$ be the variety 
in $(\C^*)^{2^n}$ consisting of all nowhere vanishing vectors of 
principal minors of matrices in $M_n (\C)$.  The ideal in
$L := \C [x_S , x^{-1}_S : S \subseteq [n]]$ of Laurent polynomials 
vanishing on $V$ is denoted $J(V)$.  Similarly, let $V'$ be the 
variety of nowhere vanishing vectors of principal and odd 
almost-principal minors of matrices in $M_n (\C)$.  Its ideal 
$J(V')$ lives in the Laurent algebra $L'$ whose variables
correspond to all vertices and faces of $\B_n$.  For these
notions and the subsequent arguments, readers not familiar
with standard algebraic notions are referred to~\cite{CLO2}.

\begin{theorem} \label{th:ideal}
The ideal $J(V')$ in $L'$ is a prime ideal generated by the 
hexahedron relations on rank-3 Boolean intervals.  The ideal 
$J(V)$ is the intersection of $J(V')$ with the sub-ring $L$.
\end{theorem}

\begin{proof}
If a set of nonzero values satisfies the hexahedron relations 
then we can construct a matrix with those principal and almost
principal minors.  Conversely any collections of principal and 
almost principal minors if all nonzero satisfies the hexahedron 
relations.  From this it is immediate that the ideal in $L'$ 
of all relations satisfied by generic vectors of minors is
the radical of $J$.  It remains to show that $J$ is prime.

Showing $J$ is prime is equivalent to showing that $L' / J$ is 
an integral domain.  The variables $x_A$ in $L'$ may be 
ordered in such a way that a sequence of hexahedron relations
beginning with the standard tiling contains precisely one
relation $x_i x_j - P$ for each non-initial variable $x_j$,
where $i < j$ and $P$ is a polynomial in $\{ x_t : t < j \}$.  
Because $x_i$ is a unit in $L'$ we may replace the generator 
wih $x_j = x_i^{-1} P$.  Recursively, we extend the $L$, the
Laurent algebra in the intial variables as follows.  Adjoin
the first non-initial variable $x_j$ and take the quotient
by $x_i x_j - P$.  The ideal $(x_i x_j = P)$ is the same as
$(x_j - x_i^{-1} P)$ so it is prime, hence $L_{1a} := 
L[x_j] / (x_i x_j - P)$ is a domain.  Next invert $x_j$:
$L_{1b} := L_1 [y_j] / (x_j y_j - 1)$.  For a similar
reason, $L_{1b}$ is a domain.  Continuing in this way,
after adjoining and inverting the last variable, we arrive
at a domain isomorphic to $L' / J_0$ where $J_0$ is generated
by the chosen hexahedron relations.  Thus $L / J_0$ is radical,
and hence $J_0$ is prime.  But $J_0 = J$: by commutativity of
the hexahedron transformations, the hexahedron relations already
imposed imply satisfaction of every hexahedron relation,
whence $J \subseteq \sqrt{J_0} = J_0 \subseteq J$ and $J$ is prime. 
\end{proof}

We have seen that the entries of $A$ are determined by 
equations in the network variables, each being linear in
the new variable, hence producing a rational function of
the initial variables.  In fact a Laurent property holds.

\begin{theorem}\label{th:laurent}
Let $A = \Psi (T_0,F)$ be the matrix such that $F_{T_0,A} = F$.
Then the entries of $A$ are Laurent polynomials in the standard 
network variables, with coefficients $1$.  The monomials in $M_{ij}$ 
are in bijection with domino tilings of the half-aztec diamond from which two squares have been removed.
\end{theorem}
 
The bijection is illustrated in Figure \ref{halfaztec4}. 
A half-aztec diamond is a region as in Figure \ref{halfaztec4} for the case $n=4$. It is a triangular stack of squares;
the bottom row consists of $2n$ squares (numbered $1$ through $2n$), and successive rows have two fewer squares.
To get the $ij$-entry of $M$ for $i\le j$, delete the 
squares on the bottom row at locations $2i-1$ and $2j$. The $M_{ij}$ entry enumerates the domino tilings of the resulting figure
using the formula of Figure \ref{halfaztec4}. 

To get the $ij$-entry of $M$ for $i>j$, delete the 
squares on the bottom row at locations $2i-1$ and $2j$ \emph{and} the outer layer of squares from the left and right sides of previous figure. 
The $M_{ij}$ entry enumerates the domino tilings of the resulting figure (a smaller half-aztec diamond)
using again the formula of Figure \ref{halfaztec4}. 

By a well-known mapping these tilings are also in a natural bijection with Schr\"oder paths (lattice paths from $(0,0)$ to $(n,n)$ using only steps 
in $\{(1,0),(1,1),(0,1)\}$ and staying at or below the diagonal).

\begin{figure}
\begin{center}\includegraphics[width=3in]{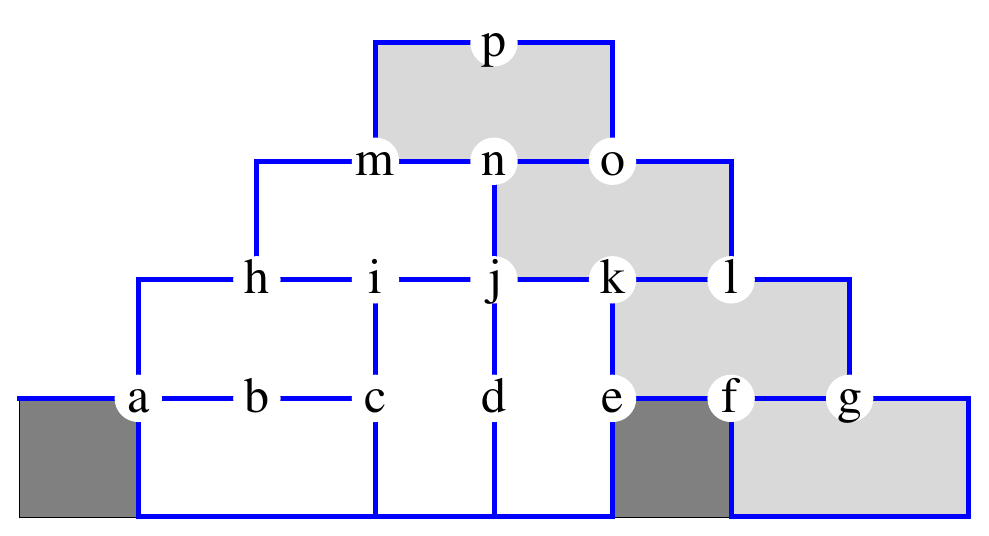}\end{center}
\caption{\label{halfaztec4}A half-aztec diamond of order $4$,
and one of the domino tilings contributing to $M_{13}$. In the corresponding monomial, each variable occurs with power $d-3$,
where $d$ is the local degree at the corresponding vertex in the domino tiling. In this example, 
the monomial is $\frac{aj}{bd}$. Note that the (light gray) tiles above the line of slope $-1$ containing the right removed square (dark gray)
are all horizontal; removing these, we have a tiling of a truncated order-$2$ diamond (the top half of an aztec diamond along with the top row of its bottom half; white in the figure).}
\end{figure}

\begin{proof} The proof uses a few facts about the combinatorics of 
Dodgson condensation, see e.g.~\cite{Spey2007}.
Recall how Dodgson condensation works. Define $m_{ij}^{(0)}=1$. 
Starting from an $n\times n$ array of numbers $(m^{(1)}_{ij})$ representing a matrix $M$.
Define an pyramidal array $m_{ij}^{(k)}$ where $i,j$ are integers for $k$ odd and half-integers for $k$ even,
by the (signed) octahedron recurrence
$$m^{(n+1)}_{i,j} = \frac{m^{(n)}_{i-1/2,j-1/2}m^{(n)}_{i+1/2,j+1/2}-m^{(n)}_{i-1/2,j+1/2}m_{i+1/2,j-1/2}^{(n)}}{m_{ij}^{(n-1)}}.$$ 
Here the defined values $m^{(k)}_{ij}$ form a pyramid called the 
\emph{Dodgson pyramid} of the matrix $M$. Its apex value is the determinant of $M$.

The (consecutive-index) principal minors of $M$ occur on a slice of the pyramid: the slice in the $x=y$ plane
(here we are thinking of the $x$-axis as the row coordinate and the $y$ axis as the column coordinate).
The almost principal minors occur on the parallel plane $x=y+1$. 

It follows from the above that the matrix $M$ associated to the standard network has entries which form the base of the Dodgson pyramid of $M$.

Typically the octahedron recurrence (for which Dodgson condensation is a special case)
is defined by taking initial data on the $z=0$ and $z=1$ planes, and working upwards.
We can, however, instead take our initial data on the $x=y$ and $x=y+1$ planes, and use the recurrence
to successively define values on planes $x-y=2,3,\dots$ and $-1,-2,\dots$. 
Because the entries of the Dodgson pyramid satisfy the \emph{signed} octahedron recurrence 
when going upward (increasing $z$),
they satisfy the \emph{unsigned} octahedron recurrence when going in these horizontal directions.

We can thus form the entries of $M$ using the octahedron recurrence 
(with $+$ signs) with initial data on the planes $x=y$ and $x=y+1$,
that is, with initial data consisting of the principal and almost principal 
minors of $M$. 

By a small generalization of a result of Speyer~\cite{Spey2007}, 
the entries on the plane $x=y+j$ are counted by domino tilings 
of truncated aztec diamonds: the entries on $z=1$ are defined by 
aztec diamonds truncated to remove all but the first row of the 
bottom half, as in the unshaded squares in Figure \ref{halfaztec4}; the entries on $z=j$ are counted 
by domino tilings of aztec diamonds from which the bottom
$n-j$ rows have been removed, that is, take the upper half and add 
the first $j$ rows of the bottom half.

To see this, extend the half-aztec diamond in the $x=y$ plane to a 
full aztec diamond, defining parameters $\eps^{-|z|}$ for vertices 
at negative $z$ values, where $\eps$ is small.  Now Speyer's 
bijection between the octahedron recurrence and tilings of the 
full aztec diamond shows that, in the limit $\eps\to 0$, the 
desired term is counting tilings of the aztec diamond in which
only horizontal dominos occur in all rows below $z=0$.  These 
are equivalent to tilings of the truncated aztec diamond.
\end{proof}


\section{Hermitian networks}

In this section we examine the image of various subsets of 
$M_n^* (\C)$ under the correspondence mapping matrices to
networks.  In particular, we describe the images of the
set of real symmetric matrices, the set of Hermitian matrices 
and the set of positive definite Hermitian matrices.
These descriptions not only parametrize the respective sets
but answer the question as to which collections of minors
are possible\footnote{Of course one answer is any Hermitian
collections of $1 \times 1$ minors and all the relations saying
that the larger minors are determinants of these.  We mean, however,
to characterize the possible vectors by identities and inequalities
involving small collections of minors.}.

\begin{definition} ~~\\[-4ex]
\begin{enumerate}[(i)]
\item A $2n$-gon network $(T,F)$ with entries in $\R$ or $\C$ is said to be
Hermitian if it satisfies the condition that $F(v)$ is real for all
vertices $v$ and for each face $f \in U(T)$ we have
\begin{equation} \label{eq:hermitian}
|F(f)|^2 = F(a) F(c) + F(b) F(d)
\end{equation}
where $a,b,c,d$ lists the vertices of $f$ in cyclic order.
\item A Hermitian network $(T,F)$ is said to be positive if 
for all vertices $v$, the sign of $F(v)$ is $\sigma (v)$.
\end{enumerate}
\end{definition}

The following result will be proved in Section~\ref{ss:hermitian corr}.
\begin{theorem} \label{th:hermitian}
The following are equivalent.
\begin{enumerate}[(i)]
\item The matrix $A \in M_n^* (\R)$ or $M_n^* (\C)$ is Hermitian; 
\item The network $(T , F_{A,T})$ is Hermitian for some $T$;
\item The network $(T , F_{A,T})$ is Hermitian for every $T$.
\end{enumerate}
\end{theorem}

\subsection{Hermitian Kashaev relation}

Values $a_0,a_4,\dots, a_9, a_0^*$ on vertices and 
$a_1, a_2, a_3, a_1^*, a_2^*, a_3^*$ 
on faces of a cube (as in Figure~\ref{cubemove}), with $a_0,a_4,\dots, a_9$ 
real, are said to satisfy the \emph{Hermitian Kashaev relation} if
(\ref{eq:hermitian}) holds on every face and

\begin{eqnarray}\label{KashaevEQ1}
a_1^*&=&\frac{a_2a_3+\bar a_1a_7}{a_0} \label{k1} \\\label{KashaevEQ2}
a_2^*&=&\frac{a_3a_1+\bar a_2a_8}{a_0} \label{k2} \\\label{KashaevEQ3}
a_3^*&=&\frac{a_1a_2+\bar a_3a_9}{a_0} \label{k3} \\\label{KashaevEQ4}
a_0^*&=&\frac{a_0a_4a_7+a_0a_5a_8+a_0a_6a_9+2a_7a_8a_9+a_1a_2a_3+\overline{a_1a_2a_3}}{a_0^2}. \label{k4} 
\end{eqnarray}

\begin{lemma}\label{lem:equiv}
Let $a_0 , \ldots , a_9$ be complex numbers making the left-hand
diagram of Figure~\ref{corr2} satisfy the relation~\eqref{eq:hermitian}
on each face:
\begin{equation} \label{eq:hermitian face}
a_1\bar a_1= a_0a_4+a_8a_9,~~~a_2\bar a_2= a_0a_5+a_7a_9,
   ~~~a_3\bar a_3= a_0a_6+a_7a_8 \, .
\end{equation}
Then the values $a_0^*, a_1^*, a_2^*, a_3^*$ obtained from a cube move (the hexahedron relation)
satisfy the Hermitian Kashaev relations~\eqref{k1}--\eqref{k4}.  
Furthermore, the right-hand side will also satisfy~\eqref{eq:hermitian}
on each face.  It follows that any network equivalent to a 
Hermitian network is Hermitian and that the Hermitian Kashaev relations 
are a special case of the hexahedron relations under the constraint 
that~\eqref{eq:hermitian face} holds on any, hence every, network.
\end{lemma}

\begin{proof}
This is a simple algebraic check: see the proof in~\cite[Section~7]{KP-ising} for real valued networks;
the same proof goes through for complex valued networks taking 
real values at the vertices.  
\end{proof}

\subsection{Hermitian correspondence} \label{ss:hermitian corr}

We begin with a proof of Theorem~\ref{th:hermitian}.  By 
Lemma~\ref{lem:equiv} $(ii)$ and $(iii)$ are equivalent.
It remains to prove that $(iii) \Rightarrow (i) \Rightarrow (ii)$.

\noindent{$(i) \Rightarrow (ii)$}:  Let $A$ be any Hermitian
matrix and let $(T_0 , F) = (T_0 , F_{T_0,A})$ be the standard
network associated with $A$.  It is immediate that $F(v)$ is
real for all vertices $v \in U(T_0)$ because principal minors
of Hermitian matrices are real.  To check~\eqref{eq:hermitian},
let $f$ be a face of $T_0$ with vertices $v, v+e_i, v+e_i+e_j$
and $v+e_j$.  Let $S \subseteq [n]$ corresond to $v$.  The values
of $F$ at the vertices of $f$ are respectively (using the notation
$M_i^j := M_{S \cup \{ i \}}^{S \cup \{ j \}}$, etc.),
$$\sigma (v) \det M\, , \;\; \sigma (v+e_i) \det M_i^i\, , \;\; 
   \sigma (v+e_i+e_j) M_{ij}^{ij}\, , \;\; \sigma (v+e_j) M_j^j \, ,$$
while the face value is $F(f) = \sigma (v+e_i) M_j^i$.  Observe that
the signs satisfy $\sigma (v) \sigma (v+e_i+e_j) = -1$.  
Dodgson's condensation \cite{Dodg1866} states that 
$$\det M_{ij}^{ij} \det M = \det M_i^i \det M_j^j - \det M_j^i \det M_i^j.$$
Because $M$ is Hermitian $M_i^j = \overline{M_j^i}$. 
Thus we have
\begin{equation} \label{TT}
-F(v + e_1 + e_2) F(v) = F(v + e_1) F(v + e_2)-|F(f)|^2
\end{equation}
which means the network is Hermitian.  

\noindent{$(iii) \Rightarrow (i)$}: Let $T$ be any tiling and  
let $f$ be a face incident to $v_0$ with sides parallel to 
$e_i$ and $e_j$.  The values of $F$ on vertices of $f$ are
$$1, A_{ii}, A_{ii} A_{jj} - A_{ij} A_{ji} , A_{jj}$$
while $F(f) = A_{ji}$.  If $(T,F_{A,T})$ is Hermitian then 
applying~\eqref{eq:hermitian} at $f$ gives
$$A_{ij} A_{ji} = |A_{ji}|^2 \, .$$
This means that $A_{ij} = \overline{A_{ji}}$.  For every $i \neq j$
there is at least one tiling $T$ having such a face $f$.  We conclude
that if each $(T , F_{A,T})$ is Hermitian then so is $A$.
$\noproof$

Fixing a tiling $T$ and assigning values of $F$ on $U(T)$ arbitrarily
(but generically) exactly parametrizes generic $n \times n$ matrices.
In the Hermitian case, the same is true if one restricts to Hermitian
networks $(T,F)$; however we would like a more explicit parametriation
of this subset of networks.

\begin{prop}\label{hermitianparam}
Generic Hermitian $n \times n$ matrices are parameterized by their
diagonal entries and contiguous almost-principal minors.
\end{prop}

\begin{proof}
We have seen that generic Hermitian matrices are parametrized
by standard networks $(T_0 , F)$ satisfying the Hermitian 
condition~\eqref{eq:hermitian}.  It remains only to observe
that these networks are parametrized by the face variables
$\{ F(f) \}$ together with the vertex variables $\{ F(e_j) :
1 \leq j \leq n \}$ along the lowest dotted path.  (To see this
note that the rhombi not adjacent to $v_0$ can be ordered so that 
each new rhombus has only one vertex not in the union of the previous
rhombi.)  If $A$ is the matrix corresonding to the network then
the values $F(e_j)$ are the diagonal elements of $A$ and the
face variables are the continguous minors 
$M_{i, \ldots , j}^{i\pm1 , \ldots , j\pm1}$
where the sign choice in the $\pm$ is determined by the parity
of $j-i$ but does not matter because the the two minors are
conjugates of each other.
\end{proof}

There are in fact many other choices of parameters. Take any shortest path $\gamma$ in the tiling from $v_0$ to the opposite vertex.
We claim that the variables on the vertices of $\gamma$, along with all face variables, parameterize all networks. To see this it suffices
to show that on either side of $\gamma$, if there is a tile (that is, if $\gamma$ is not the boundary path) there is a tile
having two consecutive sides, and thus three vertices, touching $\gamma$. The value at the fourth vertex is then a function of its
face value and the values at the vertices along the path; pushing the path across this tile and continuing, we see that 
all vertex values are obtained in this way. To find such a tile to the left (say) of $\gamma$, take any tile left of $\gamma$ 
and follow its train tracks (continguous tiles sharing a set of parallel edges) until they cross $\gamma$; take any new tile in the triangular region delimited by $\gamma$ and these two
train tracks. $\gamma$ with the train tracks of this new tile forms a strictly smaller triangular region. Conclude by induction.

Another interesting representation of the generic Hermitian matrix
is the following Laurent parametrization, where the initial conditions
are taken to be an arbitrary network.

\begin{proposition}\label{hermitianlaurent}
The matrix entries for a standard Hermitian network are Laurent polynomials 
in the interior entries (face values and interior vertex values).
\end{proposition}

\begin{proof} This follows from the essentially same argument as in the 
proof of Theorem \ref{nettomtx}.  We work outwards from the diagonal.
Inductively, each new entry $M_{ij}$ with $i<j$ is defined by an equation
$\det M_A^B = c$ where $M_A^B$ is an odd almost-principal minor.  This 
is a multilinear linear equation in which $M_{ij}$ is the only unassigned
variable; moreover the coefficient of $M_{ij}$  is a principal minor. 
Thus $M_{ij}$ is a Laurent polynomial, which is an actual polynomial 
in the (previously assigned) other matrix entries, with a denominator 
which is the parameter assigned to a principal minor, which is an interior 
vertex.
Finally we can define $M_{ji}=\overline{M_{ij}}.$
\end{proof}

\begin{example}
For the standard tiling, the $4 \times 4$ example is easy to compute.
The matrix $M$ is given in terms of the face and interior vertex variables
as follows.
\begin{equation} \label{eq:M}
M=\begin{pmatrix}
a&\bar x&\frac{\bar x\bar y+u}{b}&\frac{\overline{xyz}}{bc}+\frac{\bar z u}{bc}+\frac{\bar xv}{bc}+\frac{yuv}{bcf}+\frac{\bar w}{f}\\
x&b&\bar y&\frac{\overline{yz}+v}{c}\\
\frac{xy+\bar u}{b}&y&c&\bar z\\
\frac{xyz}{bc}+\frac{z\bar u}{bc}+\frac{x\bar v}{bc}+\frac{\overline{yuv}}{bcf}+\frac{w}{f}&\frac{yz+\bar v}{c}&z&d
\end{pmatrix}.
\end{equation}
\end{example}


Examination of the matrix entries in~\eqref{eq:M} leads to the following
conjecture, verified in the $5 \times 5$ case as well.

\begin{conjecture}
The matrix entries for a standard Hermitian network are Laurent 
polynomials, with coefficient $1$, whose numerators are monomials 
in the face variables and denominators are monomials in the 
interior vertex variables.  The terms are in bijection with 
Catalan paths on the dual network.  The purported bijection 
is illustrated in Figure \ref{catalan}. 
\end{conjecture}

\begin{figure}[ht!]
\begin{center}\includegraphics[width=2.5in]{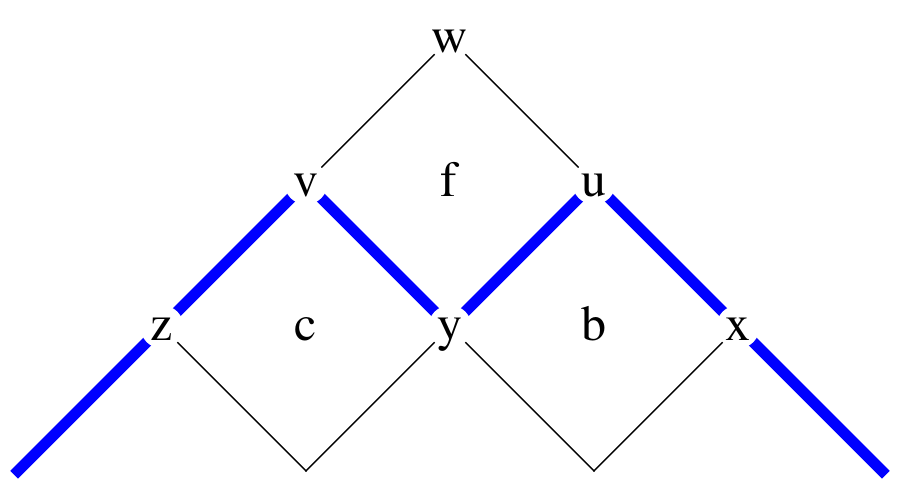}\end{center}
\caption{\label{catalan}This grid is the dual graph to the tiling. To a Catalan path on this grid, record the variables at each local max and min.
The vertex variables here (which are face variables of the network) go in the numerator; the face variables in the denominator.
The given path has weight $\frac{yuv}{bcf}$.}
\end{figure}

\subsection{Positive $2n$-gon networks}

Recall that a positive $2n$-gon network is a Hermitian network 
with the additional constraint that the sign of $F(v)$ is $\sigma(v)$.

\begin{theorem}\label{positive} The network associated to a 
positive definite Hermitian matrix is positive.  Conversely, 
a positive network gives rise to a positive-definite Hermitian matrix.
\end{theorem}

\begin{proof} Suppose a network is positive.  Under the 
network-matrix correspondence, the values along the right-hand 
boundary of the network are $\sigma(v)$ times the leading 
(upper-left principal) minors of the matrix.  Sylvester's 
criterion~\cite[Theorem~7.2.5]{horn-johnson} states that 
positivity of the leading minors of a Hermitian matrix 
(the first $k$ rows and columns, $1 \leq k \leq n$)
is equivalent to positive definiteness of the matrix. 
Thus a positive network gives rise to a positive definite matrix.

Conversely, if a matrix is positive definite, all its principal minors 
are positive and therefore the network is positive.
\end{proof}

How are positive networks parametrized?  On the standard network, 
do this as follows.  We assign values on $e_i$ arbitrarily and 
positively.  Then on $e_i+e_{i+1}$ we assign any negative values 
larger than $-F(e_i)F(e_{i+1})$, so that $F(e_i)F(e_{i+1})+F(e_i+e_{i+1})>0$.
Then on $e_i+e_{i+1}+e_{i+2}$ assign any negative value larger than 
$$F(e_i+e_{i+1})F(e_{i+1}+e_{i+2})/F(e_{i+1}).$$ 
Then on $e_i+e_{i+1}+e_{i+2}+e_{i+3}$ assign any positive value 
smaller than 
$$-F(e_i+e_{i+1}+e_{i+2})F(e_{i+1}+e_{i+2}+e_{i+3})/F(e_{i+1}+e_{i+2}),$$ 
and so on.
In each case except for the initial $e_i$ we have a bounded positive 
length open interval to choose from.

Once the vertex values have been chosen, the face values are determined 
up to a unit real or complex number.  For $\R$ there are 2 choices 
of sign for each face value.  Thus the space of positive networks 
with nonzero face values is homeomorphic to a union of $2^{n(n-1)/2}$ 
open balls each of dimension $\binom{n+1}{2}$.  For $\C$ the argument
of each face value can be chosen freely so the space of positive
Hermitian networks is homeomorphic to the product of a
$\binom{n}{2}$-torus with a $\binom{n+1}{2}$-ball (equivalently, $(\C^*)^{\binom{n}{2}} \times \C^n$).

\section{The $q$-Hermitian case}\label{quatsection}

A $q$-Hermitian matrix is a matrix of quaternions which satisfies 
$M_{ij}=(M_{ji})^*$, where $*$ denotes the quaternionic conjugate. 
The $q$-determinant of a $q$-Hermitian matrix is a real number defined by 
\be\label{qdetdef}\text{qdet} M = \sum_{\text{cycle decomps}} (-1)^{c+n}\tr M_{C_1}\tr M_{C_2}\dots\tr M_{C_k}\ee
where the sum is over cycle decompositions of $[n]$ (disregarding order), $c$ is the number of cycles, and
$\tr M_{C}$ is the trace of the product of entries in cycle $C$ (one-half the trace for cycles of length $1$ or $2$). 

For example when $a,b,c$ are real, 
$$\qdet\begin{pmatrix}a&d&e\\d^*&b&f\\e^*&f^*&c\end{pmatrix}=abc-aff^*-bee^*-cdd^*+\text{Tr}(dfe^*).$$

Dyson~\cite{dyso1970} showed that
$\text{qdet} M = \text{Pf}(Z\tilde M)$, where $Z$ is the block-diagonal 
matrix with $2\times 2$ blocks $\begin{pmatrix}0&1\\-1&0\end{pmatrix}$ 
and $\tilde M$ is the $2n\times 2n$ matrix obtained 
from $M$ by replacing each entry $M_{ij}=a+bi+cj+dk$ 
with the $2\times 2$ block $\begin{pmatrix}a+ib&c+id\\-c+id&a-ib\end{pmatrix}$. 

A $q$-Hermitian matrix is \emph{positive definite} if its leading 
minors are all positive; equivalently if its eigenvalues are 
positive~\cite{kassel-thesis}.

A $q$-Hermitian network is a network with face values in $\HH$; vertex values are real.
In each face with vertex values $a,b,c,d$ the face value $z$ satisfies $zz^*=ac+bd$. 

In the case of a $\HH$-valued Hermitian matrix, almost principal minors $\qdet M_{S\cup\{j\}}^{S\cup\{i\}}$
can also be defined, as follows \cite{dyso1970}. Instead of summing over cycle decompositions as in (\ref{qdetdef}),  one sums over decompositions
of the indices into configurations forming a path from $i$ to $j$ with the remaining indices formed into cycles. 
The contribution for a configuration is the product of traces over the cycles and
the product of the quaternions along the path.  With this definition we can define as above a $q$-Hermitian network 
associated to a $q$-Hermitian matrix
(whereas for a general matrix over $\HH$ no such definition can be made.)

\begin{theorem}\label{th:quat}
The Kashaev relation (\ref{KashaevEQ1}-\ref{KashaevEQ4}) holds when $a_1,a_2,a_3$ are quaternions (and $a_0,a_4,\dots,a_9$ are real),
for the given order of multiplication.
Theorem \ref{th:hermitian}, Propositions \ref{hermitianparam}, \ref{hermitianlaurent} and Theorem \ref{positive} hold for $q$-Hermitian matrices.
\end{theorem}

\begin{proof}
The first statement is a short check. This implies Lemma \ref{lem:equiv} via the same proof. Theorem \ref{th:hermitian}, Propositions \ref{hermitianparam} and \ref{hermitianlaurent}
then follow.
Sylvester's criterion also holds for $q$-Hermitian matrices, see 
\cite{alesker}, and thus Theorem \ref{positive} holds as well.
\end{proof}

\section*{Acknowledgements}

We are grateful to an anonymous referee for prompting us to formulate
Theorem~\ref{th:ideal} in precise terminology from algebraic geometry.
We are indebted to Bernd Sturmfels for considerable help in carrying
this out, and to Sergey Fomin and Frank Sottile for further useful
conversations.

\bibliographystyle{alpha}
\bibliography{RP}

\end{document}